\documentclass[12pt]{article}

\usepackage{xypic}
\usepackage{amsmath,amstext,amsbsy,amssymb,amsthm}
\usepackage{mine}
\begin{document}
\selectfont
\define\div{\big|}
\define\ndiv{\not\div}
\define\zetabar{\ovserline{\zeta}}
\define\ahat{\hat{a}}
\define\bhat{\hat{b}}
\define\chat{\hat{c}}
\define\dhat{\hat{d}}
\define\phat{\hat{p}}
\define\zhat{\hat{0}}
\define\vhat{\hat{v}}
\define\xhat{\hat{x}}
\define\lG{\lambda(G)}
\define\zetabar{\overline{\zeta}}
\define\half{\mathfrak{h}}
\def\keywords#1{\def\@keywords{#1}}
\parskip=0.125in
\title{Hopf-Galois Structures Arising From Groups with Unique Subgroup of Order p}
\date{}
\author{Timothy Kohl\\
Department of Mathematics and Statistics\\
Boston University\\
Boston, MA 02215\\
tkohl@math.bu.edu}
\maketitle
\begin{abstract}
For $\Gamma$ a group of order $mp$ for $p$ prime where $gcd(p,m)=1$, we consider those regular subgroups $N\leq Perm(\Gamma)$ normalized by $\lG$, the left regular representation of $\Gamma$. These subgroups are in one-to-one correspondence with the Hopf-Galois structures on separable field extensions $L/K$ with $\Gamma=Gal(L/K)$. This is a follow up to the author's earlier work where, by assuming $p>m$, one has that all such $N$ lie within the normalizer of the $p$-Sylow subgroup of $\lambda(\Gamma)$. Here we show that one only need assume that all groups of a given order $mp$ have a unique $p$-Sylow subgroup, and that $p$ not be a divisor of the automorphism groups of any group of order $m$. As such, we extend the applicability of the program for computing these regular subgroups $N$ and concordantly the corresponding Hopf-Galois structures on separable extensions of degree $mp$.
\end{abstract}
\noindent {\it Key words:} Hopf Galois extension, regular subgroup \par
\noindent {\it MSC:} 20B35, 20D20, 20D45, 16T05\par
\renewcommand{\thefootnote}{}
\centerline{Introduction}
The motivation and antecedents for this lie in the subject of Hopf Galois theory for separable field extensions. In particular this note is about extending the recent work \cite{Kohl2013} by the author on Hopf Galois structures on Galois extensions of degree $mp$ for $p$ a prime where $p>m$. We will not delve into all the particulars of Hopf Galois theory, since this discussion focuses on the group theoretic underpinnings of this class of examples. For the general theory, one may consult references such as \cite{ChaseSweedler1969} (for the basic definitions and initial examples) and \cite{GreitherPareigis1987} for the theory as applied to separable extensions which is the category in which \cite{Kohl2013} and others fall.  In brief, let $L/K$ be a finite Galois extension with $\Gamma=Gal(L/K)$. Such an extension is canonically Hopf Galois for the $K$-Hopf algebra $H=K[\Gamma]$, but also for potentially many other $K$-Hopf algebras. The enumeration of these is determined by the following variant of the main theorem in Greither and Pareigis which we paraphrase here:\par
\noindent{\bf Theorem: }\cite{GreitherPareigis1987}
If $L/K$ is a finite Galois extension with $\Gamma=Gal(L/K)$ then the Hopf algebras which act to make the extension Hopf Galois correspond in a one to one fashion with the regular (transitive and fixed point free) subgroups $N$ of $B=Perm(\Gamma)$ such that $\lambda(\Gamma)\leq Norm_B(N)$.\par
Each such $N$ gives rise to the Hopf algebra $H=L[N]^{\Gamma}$ the fixed ring of the group ring $L[N]$ under the action of $\Gamma$ simultaneously on the coefficients, by virtue of $\Gamma=Gal(L/K)$, and the group elements by virtue of $\lambda(\Gamma)$ normalizing $N$. The problem of enumerating such $N$ for different classes of extensions has been the subject of much recent work by Byott (e.g. \cite{Byott2004}), Childs (e.g. \cite{Childs2003}), the author and others.\par
In order to organize the enumeration of the possible $N$ that may arise for a given $\Gamma$, one considers, for $[M]$ the isomorphism class of a given group of the same order as $\Gamma$, the set:
$$R(\Gamma,[M])=\{N\leq B\ |\ N\ regular, N\cong M , \lambda(\Gamma)\leq Norm_B(N)\}$$
and let $R(\Gamma)$ be the union of the $R(\Gamma,[M])$ over {\it all} isomorphism classes of groups of the same order as $\Gamma$. The totality of all of these give all possible Hopf Galois structures that exist on the extension $L/K$.\par
Again, in the cases considered in \cite{Kohl2013} it was assumed that $|\Gamma|=mp$ for $p$ prime where $p>m$. Our goal is to extend those results to groups of order $mp$ where $gcd(p,m)=1$, but where one need not assume that $p>m$.
\par
\section{Preliminaries}
We begin by briefly reviewing the setup in \cite{Kohl2013} where the author considered groups $\Gamma$ of order $mp$ with a unique and therefore characteristic $p$-Sylow subgroup due to the assumption that $p>m$. Since $p>m$ obviously implies $gcd(p,m)=1$ then by the Schur-Zassenhaus lemma, $\Gamma$ may be written as $PQ$ for $P$ and $Q$ subgroups of $\Gamma$ where $|P|=p$ and $|Q|=m$. More specifically there is a split exact sequence $P \rightarrow \Gamma\rightarrow Q$ whereby $\Gamma=P\rtimes_{\tau}Q$ with $\tau:Q\rightarrow Aut(P)$ is induced by conjugation within $\Gamma$ by the complementary subgroup $Q$. Using $Q$ for the quotient of $\Gamma$ by $P$ and the image of the section in $\Gamma$ is admittedly a slight abuse of notation. The condition $p>m$ is sufficient, of course, to make the $p$-Sylow subgroup unique and have order $p$. Going forward, we wish to drop the assumption that $p>m$ and consider groups of order $mp$ for $p$ prime, with $gcd(p,m)=1$ and where congruence conditions force {\it any} group of order $mp$ to have a unique $p$-Sylow subgroup.\par
\noindent If $\Gamma$ has order $mp$ for $gcd(p,m)=1$ and has a unique $p$-Sylow subgroup then if $\lambda:\Gamma\rightarrow Perm(\Gamma)=B$ is the left regular representation then we define $\mathcal{P}=P(\lambda(\Gamma))$ to be the $p$-Sylow subgroup of $\lambda(\Gamma)$ and $\mathcal{Q}$ to be the complementary subgroup to $\mathcal{P}$ in $\lambda(\Gamma)$. For $p>m$, the program in sections 1-3 of \cite{Kohl2013} is based on the following core result:
\begin{theorem}\cite[Theorem 3.5]{Kohl2013}
\label{main}
For $\Gamma$ of order $mp$ with $p>m$, if $N\in R(\Gamma)$ then $N$ is a subgroup of $Norm_B(\mathcal{P})$.
\end{theorem}
\noindent To extend this to groups of order $mp$ with $gcd(p,m)=1$, we need to modify certain key results from \cite{Kohl2013}, starting with \cite[Lemma 1.1]{Kohl2013} regarding the $p$-torsion of $Aut(\Gamma)$. 
\begin{lemma}
\label{Aut}
Let $\Gamma$ have order $mp$ where $gcd(p,m)=1$ which has a unique $p$ Sylow subgroup of order $p$, where $\Gamma\cong P\rtimes_{\tau}Q$ as above. Then\par
\noindent (a) If $\tau$ is trivial (whence $\Gamma\cong P\times Q$) then $p$ does not divide $|Aut(\Gamma)|$.\par
\noindent (b) If $\tau$ is non-trivial then $Aut(\Gamma)$ has a unique $p$-Sylow subgroup, consisting of inner automorphism induced by conjugation by elements of $P$, \par
\noindent provided that $p$ does not divide $|Aut(Q)|$.\par
\end{lemma}
\begin{proof}
In (a), if $\Gamma$ is such a direct product then $Aut(\Gamma)=Aut(P)\times Aut(Q)$ and so if $p\nmid |Aut(Q)|$ then $p\nmid |Aut(\Gamma)|$. The proof is basically the same as in \cite{Kohl2013}. For (b), since $\Gamma/P\cong Q$ then $\psi\in Aut(\Gamma)$ induces $\bar{\psi}\in Aut(\Gamma/P)\cong Aut(Q)$ and if $\psi$ has $p$-power order then $\bar{\psi}$ is trivial if $p$ does not divide $|Aut(Q)|$. And, as also observed in \cite{Kohl2013}, when $\Gamma$ is not a direct product, then $|P\cap Z(\Gamma)|=1$ and conjugation in $\Gamma$ by elements of $P$ yields the order $p$ subgroup of $Aut(\Gamma)$.
\end{proof}
Note, the condition that $p\nmid |Aut(Q)|$ was automatic when $p>m$, but, as it turns out, this holds true fairly often even when $p<m$. For example, if $p=5$ and $m=8$ then Sylow theory easily shows that any group of order 40 will have a unique $p$-Sylow subgroup. Moreover one may consider each group of order 8, $\{C_8,C_4\times C_2,C_2\times C_2\times C_2,D_4,Q_2\}$ whose automorphism groups have orders $\{4,8,168,8,24\}$ respectively, {\it none} of which are divisible by $5$.\par

\noindent As regularity is so essential to this discussion, we define it here:
\begin{definition}
If $X$ is a finite set and $B=Perm(X)$ then a $N\leq B$ is regular if any two of the following properties hold
\begin{enumerate}
\item $N$ acts transitively on $X$
\item $\eta(x)=x$ for $\eta\in N$ only if $\eta=e_N$, the identity element of $N$
\item $|N|=|X|$.
\end{enumerate}
\end{definition}
\noindent Property (2) above is particularly important for understanding the cycle structure of elements in a regular permutation group. In particular, any non-identity element of a regular permutation group must be a product of cycles of the same length where the sum of the lengths must equal $|X|$. For example, if $X=\{1,2,3,4,5,6\}$ then $(1,2)(3,4)(5,6)$ and $(1,2,3)(4,5,6)$ satisfy this property. In contrast, $\mu=(1,2,3,4)(5,6)$ does not belong to any regular subgroup of $Perm(X)$ even though it acts fixed point freely, the difference being that $\mu^2=(1,3)(2,4)$ which has fixed points. i.e. If $\eta$ belongs to a regular permutation group then not only must $\eta$ act fixed point freely, but also any non-trivial power of $\eta$. A subgroup of $B$ which satisfies condition (2) by itself is termed {\it semiregular} (where of course any semiregular subgroup of size equal to $|X|$ is regular). Moreover, any subgroup of a regular permutation group is semiregular, in particular $\mathcal{P}\leq\lambda(\Gamma)$. As such, it must be generated by an element $\pi=\pi_1\pi_2\cdots\pi_m$ where the $\pi_i$ are disjoint $p$-cycles. In a similar fashion, if $N$ is any regular subgroup of $B$ then it's $p$-Sylow subgroup $P(N)$ is cyclic of order $p$ and similarly generated by a product of $m$ $p$-cycles. For those regular $N\leq B$ corresponding to Hopf-Galois structures where the {\it underlying} group is $\Gamma$ we are looking at those which are normalized by $\lambda(\Gamma)$.

The $p>m$ assumption was used in \cite{Kohl2013} to show that for any such $N$, that $P(N)=\langle\pi_1^{a_1}\cdots\pi_m^{a_m}\rangle$ where each $a_i\in U_p=\mathbb{F}_p^{\times}$. This was due to the observation \cite[Prop 1.2]{Kohl2013} that if $p>m$ then $P(N)=\langle\theta\rangle$ and $\mathcal{P}=\langle\pi\rangle$ must centralize each other forcing $\theta\pi_i\theta^{-1}=\pi_i$ (because for $p>m$ the group $S_m$ contains no elements of order $p$) and consequently that $\theta$ is a product of the same $\pi_i$ (that comprise the generator of $\mathcal{P})$ to non-zero powers.\par
\noindent As it turns out, this is {\it not} automatically true if it's merely assumed that $gcd(p,m)=1$. For example, if $p$=5 and $m$=8 then in $S_{40}$ let 
$$\pi_i=(1+(i-1)5,2+(i-1)5,3+(i-1)5,4+(i-1)5,5+(i-1)5)$$
for $i=1,\dots,8$ and let $\theta_j=(j,j+5,j+10,j+15,j+20)$ for $j=1,\dots,5$ and $\theta_6=\pi_6$, $\theta_7=\pi_7$, $\theta_8=\pi_8$. One may verify that $\pi=\pi_1\cdots\pi_8$ is centralized by $\theta=\theta_1\cdots\theta_8$ but for $j=1,\dots,5$ that $\theta_j$ is not a power of any $\pi_i$.\par
This example shows that the $P(N)\leq N$ being normalized, and thus centralized, by $\mathcal{P}$ is insufficient to guarantee that $P(N)\leq \mathcal{V}=\langle \pi_1,\pi_2,\dots,\pi_m\rangle$. This does not nullify the possibility of the program in \cite{Kohl2013} being generalized to the case $gcd(p,m)=1$. This example merely shows that $Cent_B(\mathcal{P})$ contains many semi-regular subgroups of order $p$ that are not subgroups of $\mathcal{V}$. However, it turns out that for those $N$ normalized by $\lambda(\Gamma)$ that (since $P(N)$ is characteristic in $N$ and therefore normalized by $\lambda(\Gamma)$) the possibilities for $P(N)$ are still restricted to being contained in $\mathcal{V}$. To arrive at this we need to recall some facts about the structure of $Norm_B(\mathcal{P})$ and $Cent_B(\mathcal{P})$.\par
Since $\mathcal{P}=\langle\pi\rangle=\langle \pi_1\cdots\pi_m\rangle$ where the $\pi_i$ are disjoint $p$-cycles then we may choose $\gamma_i\in\Gamma$ for $i=1,\dots,m$ such that $\pi_i=(\gamma_i,\pi(\gamma_i),\dots,\pi^{p-1}(\gamma_i))$ and if we let $\Pi_i=Supp(\pi_i)$ the support of $\pi_i$ then the $\Pi_i$ are, of course, disjoint and their union is $\Gamma$ as a set. With this analysis of the cycle structure of $\pi$ in place, we pause to give the following definition for understanding the factorization of similar fixed point free elements of order $p$ in $B$.
\begin{definition}
For $\theta\in B$ and $\pi_i$ as above, we say $\pi_i$ divides $\theta$ denoted $\pi_i | \theta$ if the cycle structure of $\theta$ contains some non-trivial power of $\pi_i$. Similarly we write $\pi_i\nmid\theta$ if no power of $\pi_i$ is a factor in the cycle structure of $\theta$.
\end{definition}
\noindent Observe that if $\pi_i|\theta$ then $\pi_i|\theta^e$ for any $e\in U_p$.\par
We recall certain facts about $Cent_B(\mathcal{P})$ and $Norm_B(\mathcal{P})$ as given in section 3 of \cite{Kohl2013}. Define $\mathcal{S}\leq B$ to be those permutations $\overline{\alpha}$ such that for each $i\in\{1,\dots,m\}$ there exists a single $j\in\{1,\dots,m\}$ such that $\overline{\alpha}(\pi^t(\gamma_i))=\pi^t(\gamma_j)$ for each $t\in\mathbb{Z}_p$. Equivalently $\overline{\alpha}$ operates on the {\it blocks} $\Pi_i$ as follows $\overline{\alpha}(\{\gamma_i,\pi(\gamma_i),\dots,\pi^{p-1}(\gamma_i)\})=\{\gamma_j,\pi(\gamma_j),\dots,\pi^{p-1}(\gamma_j)\})$. It is clear that $\mathcal{S}$ is isomorphic to $S_m$ viewed as $Perm(\{\Pi_1,\dots,\Pi_m\})$ where $\overline{\alpha}\in\mathcal{S}$ corresponds to a permutation $\alpha\in S_m$ which permutes the $m$ blocks $\Pi_i$ amongst each other. In a similar fashion we may define another subgroup $\mathcal{U}\leq B$ keyed to $\pi$ and the $\pi_i$. For a cyclic group $C=\langle x\rangle$ the automorphisms are given by $x\mapsto x^c$ for $c\in U_p=\mathbb{F}_p^{\times}=\langle u\rangle$. Within $B$ therefore, since $\mathcal{P}$ is cyclic of order $p$ there exists $u_1\cdots u_m$ a product of $m$ disjoint $p-1$ cycles with the property that $u_i\pi_i u_i^{-1}=\pi_i^{u}$ so that $(u_1\cdots u_m)\pi(u_1\cdots u_m)^{-1}=\pi^{u}$ and we define $\mathcal{U}=\langle u_1\cdots u_m\rangle$. With this in mind, we may easily describe $Cent_B(\mathcal{P})$ and $Norm_B(\mathcal{P})$ (as in \cite{Kohl2013}) specifically, if $C_p$ denotes the cyclic group order $p$ and $S_m$ then
\begin{align*}
Cent_B(\mathcal{P}) &= \mathcal{V}\mathcal{S}\cong C_p\wr S_m\cong C_p^m \rtimes S_m \\ 
Norm_B(\mathcal{P}) &= \mathcal{V}\mathcal{U}\mathcal{S}\cong C_p^m \rtimes (Aut(C_p) \times S_m) \\
\end{align*}
The semi-direct product formulation is useful and may be closely connected to the intrinsic (as a subgroup of $B$) description. We may view $\mathcal{V}=\langle\pi_1,\dots,\pi_m\rangle$ naturally as $C_p^m$ but also, more perspicaciously, as $\mathbb{F}_p^m$, the dimension $m$ vector space over $\mathbb{F}_p$ so that we may equate $\pi_1^{a_1}\cdots\pi_m^{a_m}$ with $[a_1,\dots,a_m]$ a 'vector' in $\mathbb{F}_p^m$. As the group $\mathcal{S}$ permutes the $\pi_i$ amongst themselves, then we may identify it with permutations $\alpha\in S_m$ acting by coordinate shift on the vectors $\ahat=[a_1,\dots,a_m]$ and $u\in U_p$ acts by scalar multiplication. As such, we may represent a typical element of $Norm_B(\mathcal{P})$ by a triple $(\hat{a},u^r,\alpha)$ where $\ahat\in\mathbb{F}_p^m$, $u\in U_p$ and $\alpha\in S_m$ where (as a permutation) $(\ahat,u^r,\alpha)(\pi_i^k(\gamma_i))=\pi_{\alpha(i)}^{ku^r+a_{\alpha(i)}}(\gamma_{\alpha(i)})$ and where multiplication (and resulting conjugation operations) is defined by 
\begin{align*}
(\ahat,u^r,\alpha)(\bhat,u^s,\beta)&=(\ahat+u^s\alpha(\bhat),u^{r+s},\alpha\beta)\\
(\bhat,u^s,\beta)(\ahat,u^r,\alpha)&(\bhat,u^s,\beta)^{-1}=\\
       &(\bhat+u^s\beta(\ahat)-u^r(\beta\alpha\beta^{-1})(\bhat),u^{r},\beta\alpha\beta^{-1})\\
(\ahat,u^r,\alpha)^n&=(\sum_{t=0}^{n-1}u^{rt}\alpha^t(\ahat),u^{rn},\alpha^n)\\
\end{align*}
and where the elements of $\mathcal{V}$ correspond to tuples of the form $(\vhat,1,I)$ where $I$ is the identity of $S_m$, in particular $\pi=\pi_1\pi_2\cdots\pi_m=([1,1,\dots,1],1,I)$. 
The elements of $Cent_B(\mathcal{P})$ correspond to those tuples where $r=0$, (i.e. the middle coordinate is 1) which leads us back to the discussion of $P(N)$ for $N$ a regular subgroup of $B$ normalized by $\lambda(\Gamma)$. In this situation we have $P(N)=\langle\theta\rangle$ where $\theta=(\ahat,1,\alpha)$ is fixed point free and order $p$. As such, if $P(N)\not\leq \mathcal{V}$ then $\alpha\neq I$ and so $\alpha$ has order $p$ which, since $\alpha\in S_m$ for $gcd(m,p)=1$ means that $\alpha$ has fixed points in $\{1,\dots,m\}$. If $\alpha(i)=i$ then $\theta(\pi_i^k(\gamma_i))=\pi_{\alpha(i)}^{k+a_{\alpha(i)}}(\gamma_{\alpha(i)})=\pi_i^{k+a_i}(\gamma_i)$ which means that $a_i\neq 0$ and more importantly that $\theta$ restricted to ${\Pi_i}$ equals $\pi_i^{a_i}$ (i.e. $\pi_i|\theta$). And for those $j$ {\it not} fixed by $\alpha$ that $\theta$ restricted to ${\Pi_j}$ is not a power of $\pi_j$ (i.e. $\pi_j\nmid \theta$). That is, $\theta=\theta_1\theta_2\cdots\theta_m$ where $\theta_i=\pi_i^{a_i}$ only for those $i$ fixed by $\alpha$ of which there must be at least one. But since $gcd(p,m)=1$ not all will be fixed and so at least one $\theta_j\not\in\mathcal{V}$. The example given above for $S_{40}$ is an instance of this, in particular the fixed point free element of order $5$ is $([1,1,1,1,1,1,1,1],1,(1,2,3,4,5))\in Cent_{S_{40}}(\langle\pi_1\pi_2\cdots\pi_8)\rangle$.\par
Now, the requirement that $N$ be normalized by $\lambda(\Gamma)$ together with the fact that $P(N)$ is characteristic means that $P(N)$ is normalized by $\lambda(\Gamma)$. The upshot of this is the following recapitulation of \cite[Prop 1.2]{Kohl2013}.
\begin{proposition}
\label{semiregular}
For $N$ a regular subgroup of $B$ normalized by $\lambda(\Gamma)$ if $P(N)$ is the $p$-Sylow subgroup of $N$ then $P(N)$ is a semi-regular subgroup of $\mathcal{V}=\langle \pi_1,\pi_2,\dots,\pi_m\rangle$. That is $P(N)=\langle\pi_1^{a_1}\cdots\pi_m^{a_m}\rangle$ where each $a_i\in U_p=\mathbb{F}_p^{\times}$.
\end{proposition}
\begin{proof}
If $P(N)=\langle \theta\rangle$ is not a subgroup of $\mathcal{V}$ then as shown above $\theta=\theta_1\theta_2\cdots\theta_m$ where for some $i$, $\theta_i=\pi_i^{a_i}$ (i.e. $\pi_i |\theta$) and for some $j\neq i$, $\pi_j \nmid\theta$. Now, as $\lambda(\Gamma)$ certainly normalizes $\mathcal{P}=\langle\pi_1\cdots\pi_m\rangle$, then since $\lambda(\Gamma)$ contains elements $(\bhat,u^s,\beta)$ of order coprime to $p$ and $\lambda(\Gamma)$ has no fixed points, then by \cite[Prop. 3.8]{Kohl2013} $\beta\in S_m$ is fixed point free. In fact, if $\mathcal{Q}$ is the complementary subgroup of $\mathcal{P}$ in $\lambda(\Gamma)$ and $t:(\bhat,u^s,\beta)\mapsto \beta$ is projection onto the permutation coordinate, then $t(\mathcal{Q})$ is a regular subgroup of $S_m$! As such we may pick an element $\mathfrak{g}=(\bhat,u^s,\beta)$ in $\mathcal{Q}$ such that $\beta(i)=j$  
and so $\mathfrak{g}([1,1,\dots,1],1,I)\mathfrak{g}^{-1}=(u^s[1,1,\dots,1],1,I)$ where, in particular 
$$
\mathfrak{g}\pi_1\pi_2\cdots\pi_m\mathfrak{g}^{-1}=\pi_{\beta(1)}^{u^s}\pi_{\beta(2)}^{u^s}\cdots\pi_{\beta(m)}^{u^s}
$$
And since $\mathfrak{g}(\theta_1\theta_2\cdots\theta_m)\mathfrak{g}^{-1}=(\mathfrak{g}\theta_1 \mathfrak{g}^{-1})(\mathfrak{g}\theta_2 \mathfrak{g}^{-1})\cdots(\mathfrak{g}\theta_m\mathfrak{g}^{-1})$ then $\mathfrak{g}\theta_i \mathfrak{g}^{-1}=\mathfrak{g}\pi_i^{a_i}\mathfrak{g}^{-1}=\pi_{\beta(i)}^{u^sa_i}=\pi_j^{u^sa_i}$. As such, $\pi_j | \mathfrak{g}\theta\mathfrak{g}^{-1}$. The problem now is that $\mathfrak{g}\theta\mathfrak{g}^{-1}=\theta^e$ for some $e\in U_p$ implies that $\pi_j | \theta$ contrary to the assumption that $\pi_j\nmid \theta$. We therefore conclude that any such $\theta$ must be a fixed point free subgroup of $\mathcal{V}$ and therefore be of the form asserted in the statement of the proposition.
\end{proof}
With \ref{Aut} and \ref{semiregular} in place, the rest of the program, in particular the characterization of the possibilities for $P(N)$ determined by linear characters $\mathcal{Q}\rightarrow\mathbb{F}_p^{\times}$ (\cite[Theorem 2.1]{Kohl2013} and \cite[Lemma 2.3]{Kohl2013}) all follow naturally. The main theorem,\ref{main} above, that all such $N$ are subgroups of $Norm_B(\mathcal{P})$ also follows without any modification. The reason for this is that none of these subsequent results require $p>m$, merely that $gcd(p,m)=1$ and that groups of order $mp$ have unique $p$-Sylow subgroup and that $p$ not divide the automorphism group of the complementary subgroup of order $m$. To be slightly formal, if $n_p$ denotes the number of $p$-Sylow subgroups of a group, we define the following subsets of $\mathbb{N}\times\mathbb{N}$
\begin{align*}
F_{Q}&=\{(p,m)\ |\ p\text{ prime },gcd(p,m)=1, p\nmid |Aut(Q)|\text{ for all groups $Q$ of order $m$}\} \\
F_{S}&=\{(p,m)\ |\ p\text{ prime },gcd(p,m)=1, n_p=1\text{ for all groups of order $mp$}\}
\end{align*}
As such, the program in \cite{Kohl2013} for enumerating Hopf-Galois structures on Galois extensions of order $mp$ may be used for those $(p,m)\in F_{Q}\cap F_{S}$. As in \cite{Kohl2013}, $(p,m)\in F_Q\cap F_S$ for $p$ prime when $p>m$, but we want to now consider other $p$ and $m$. The case of $p=5$ and $m=8$ as indicated already is one such example. In lieu of working out the enumeration of all the $14^2$ possible pairings $R(\Gamma,[N])$ for order 40 we shall instead conclude with an overview of some the choices for $|\Gamma|=|N|=n=pm$ which 'force' $(p,m)\in F_Q\cap F_S$. Such forcing conditions have appeared in group theory literature including recent examples such as \cite{Nilpotent2000}. Our example will be somewhat more narrow, but still in this same spirit.
\section{Groups of order a product of three primes}
If $p_1<p_2<p_3$ are primes then it is a standard exercise in many group theory textbooks to show that at least one of $n_{p_2}$, $n_{p_3}$ must be 1 where the $n_{p_i}$ denote the number of $p_i$-Sylow subgroups. As such there is guaranteed to be a unique Sylow subgroup of order $p_i$, (i.e. $(p_i,p_jp_k)\in F_S$). As such the complementary subgroup $Q$ is either a cyclic or metacyclic group of order $p_jp_k$ for $p_j < p_k$. This does not preclude the possibility of course of $n_{p_1}=1$. The question then is whether $(p_i,p_jp_k)\in F_Q$ as well. Let us examine the three possible cases depending on which Sylow number is 1. 
\begin{align*}
n_{p_1}&=1\rightarrow  |Aut(Q)|=\begin{cases} (p_2-1)(p_3-1) \ Q\text{ abelian } \\ p_3(p_3-1) \ Q\text{ non-abelian }\end{cases} \\
n_{p_2}&=1\rightarrow  |Aut(Q)|=\begin{cases} (p_1-1)(p_3-1) \ Q\text{ abelian } \\ p_3(p_3-1) \ Q\text{ non-abelian }\end{cases} \\
n_{p_3}&=1\rightarrow  |Aut(Q)|=\begin{cases} (p_1-1)(p_2-1) \ Q\text{ abelian } \\ p_2(p_2-1) \ Q\text{ non-abelian }\end{cases} \\
\end{align*}
So when $n_{p_3}=1$ then $p_3 \nmid |Aut(Q)|$ regardless of whether $Q$ is abelian or not. For the other two possibilities, further restrictions on the choices of the $p_i$ are needed since $|Aut(Q)|$ can indeed be divisible by that $p$ such that $n_p=1$. Also, it is not impossible that two or more of the $n_{p_i}$ may be 1 simultaneously. We can test various tuples of primes $(p_1,p_2,p_3)$ using the naive congruence conditions from Sylow theory that force one (or more) of the $n_{p_i}$ to be 1, and also the above orders of potential $Aut(Q)$, where it's necessary and sufficient to avoid having $n_{p_i}\equiv 1(mod\ p_k)$ in order for $(p_i,p_jp_k)\in F_Q$ regardless of whether a non-abelian $Q$ of order $p_jp_k$ exists. 

 We include a sample of these tuples (in dictionary order), including, most notably, those where $p<m$. We also include repeats of those triples where more than one $n_{p_i}=1$ is forced by congruence conditions. \par
\begin{tabular}{ll}
\begin{tabular}{|c|c|c|||c|c|c|c|c|}\hline
$p_1$ & $p_2$ & $p_3$ & $p$ & $m$ & $mp$ & $p<m ?$ \\ \hline
2 & 3 & 7 & 7 & 6 & 42 &  \\ \hline 
2 & 3 & 11 & 11 & 6 & 66 &  \\ \hline 
2 & 3 & 13 & 13 & 6 & 78 &  \\ \hline 
2 & 3 & 17 & 17 & 6 & 102 &  \\ \hline 
2 & 3 & 19 & 19 & 6 & 114 &  \\ \hline 
2 & 3 & 23 & 23 & 6 & 138 &  \\ \hline 
2 & 3 & 29 & 29 & 6 & 174 &  \\ \hline 
2 & 5 & 7 & 5 & 14 & 70 & *  \\ \hline 
2 & 5 & 7 & 7 & 10 & 70 & *  \\ \hline 
2 & 5 & 11 & 11 & 10 & 110 &  \\ \hline 
2 & 5 & 13 & 13 & 10 & 130 &  \\ \hline 
2 & 5 & 17 & 5 & 34 & 170 & *  \\ \hline 
2 & 5 & 17 & 17 & 10 & 170 &  \\ \hline 
2 & 5 & 19 & 5 & 38 & 190 & *  \\ \hline 
2 & 5 & 19 & 19 & 10 & 190 &  \\ \hline 
2 & 5 & 23 & 23 & 10 & 230 &  \\ \hline 
2 & 5 & 29 & 5 & 58 & 290 & *  \\ \hline 
2 & 5 & 29 & 29 & 10 & 290 &  \\ \hline 
2 & 7 & 11 & 11 & 14 & 154 & *  \\ \hline 
2 & 7 & 13 & 7 & 26 & 182 & *  \\ \hline 
2 & 7 & 17 & 7 & 34 & 238 & *  \\ \hline 
2 & 7 & 17 & 17 & 14 & 238 &  \\ \hline 
2 & 7 & 19 & 7 & 38 & 266 & *  \\ \hline 
2 & 7 & 19 & 19 & 14 & 266 &  \\ \hline 
2 & 7 & 23 & 7 & 46 & 322 & *  \\ \hline 
2 & 7 & 23 & 23 & 14 & 322 &  \\ \hline 
2 & 7 & 29 & 29 & 14 & 406 &  \\ \hline 
2 & 11 & 13 & 11 & 26 & 286 & *  \\ \hline 
2 & 11 & 13 & 13 & 22 & 286 & *  \\ \hline 
2 & 11 & 17 & 17 & 22 & 374 & *  \\ \hline 
\end{tabular}
&
\begin{tabular}{|c|c|c|||c|c|c|c|c|}\hline
$p_1$ & $p_2$ & $p_3$ & $p$ & $m$ & $mp$ & $p<m ?$ \\ \hline
2 & 11 & 19 & 11 & 38 & 418 & *  \\ \hline 
2 & 11 & 19 & 19 & 22 & 418 & *  \\ \hline 
2 & 11 & 23 & 23 & 22 & 506 &  \\ \hline 
2 & 11 & 29 & 11 & 58 & 638 & *  \\ \hline 
2 & 11 & 29 & 29 & 22 & 638 &  \\ \hline 
2 & 13 & 17 & 13 & 34 & 442 & *  \\ \hline 
2 & 13 & 17 & 17 & 26 & 442 & *  \\ \hline 
2 & 13 & 19 & 13 & 38 & 494 & *  \\ \hline 
2 & 13 & 19 & 19 & 26 & 494 & *  \\ \hline 
2 & 13 & 23 & 13 & 46 & 598 & *  \\ \hline 
2 & 13 & 23 & 23 & 26 & 598 & *  \\ \hline 
2 & 13 & 29 & 13 & 58 & 754 & *  \\ \hline 
2 & 13 & 29 & 29 & 26 & 754 &  \\ \hline 
2 & 17 & 19 & 17 & 38 & 646 & *  \\ \hline 
2 & 17 & 19 & 19 & 34 & 646 & *  \\ \hline 
2 & 17 & 23 & 17 & 46 & 782 & *  \\ \hline 
2 & 17 & 23 & 23 & 34 & 782 & *  \\ \hline 
2 & 17 & 29 & 17 & 58 & 986 & *  \\ \hline 
2 & 17 & 29 & 29 & 34 & 986 & *  \\ \hline 
2 & 19 & 23 & 19 & 46 & 874 & *  \\ \hline 
2 & 19 & 23 & 23 & 38 & 874 & *  \\ \hline 
2 & 19 & 29 & 29 & 38 & 1102 & *  \\ \hline 
2 & 23 & 29 & 23 & 58 & 1334 & *  \\ \hline 
2 & 23 & 29 & 29 & 46 & 1334 & *  \\ \hline 
3 & 5 & 11 & 11 & 15 & 165 & *  \\ \hline 
3 & 5 & 13 & 5 & 39 & 195 & *  \\ \hline 
3 & 5 & 13 & 13 & 15 & 195 & *  \\ \hline 
3 & 5 & 17 & 17 & 15 & 255 &  \\ \hline 
3 & 5 & 19 & 5 & 57 & 285 & *  \\ \hline 
3 & 5 & 19 & 19 & 15 & 285 &  \\ \hline 
\end{tabular}
\end{tabular}\par
\noindent Before going further, we note two tuples not in this list, namely $(2,3,5)$ and $(3,5,7)$. Both lie in $F_S\cap F_Q$ but are not forced to lie in $F_S$ by the basic congruence condition $n_p\equiv 1(mod\ p)$ but rather by being three distinct primes as mentioned above. What's interesting is that these are the only two in this sorted list which require one to fall back on the three primes condition. If one wanted to, one could use something like the {\tt AllSmallGroups} library together with the {\tt AutomorphismGroup} function in GAP (\cite{GAP4}) to determine which tuples give rise to $(p,m)\in F_{Q}\cap F_{S}$ by brute force checking of known small groups of order $p_1p_2p_3$.\par
As can be seen, there are quite a number of examples where $(p,m)\in F_{Q}\cap F_{S}$ and $(p',m')\in F_{Q}\cap F_{S}$ where $mp=m'p'$ for which both $p<m$ and $p'<m'$ or where $p>m$ but $p'<m'$. What is also interesting to note about these repeats is that the whole program in \cite{Kohl2013} can computed based on either choice of $p$, the difference being the nature of the resulting groups $Q$ as well as the ambient $Norm_B(\mathcal{P})$ containing all the elements of $R(\Gamma_i,[\Gamma_j])$. For many of the examples listed in the table for these duplicates, the number of groups of order $m$ and $m'$ are the same, however this is not always the case. Consider $(p,m)=(5,39)$ versus $(p',m')=(13,15)$ and that there are 2 groups of order $39$ but only 1 of order 15. As such, for groups of order 195, the application of the program is different for $(13,15)$ since one only needs to work with one complementary group $Q$. Of course there are only $2$ groups of order 195 anyway, but it seems likely that for other $|\Gamma|=pm$ (not necessarily a product of three primes) for multiple $(p,m)\in F_Q\cap F_S$ that one may choose that $m$ such that the number of groups of order $m$ is minimal. Of course, for larger $m$, $Norm_B(\mathcal{P})$ will be larger since its order is $p^m\phi(p)m!$. However the size of $Norm_B(\mathcal{P})$ is not an obstruction in the implementation of the program in \cite{Kohl2013} except if one wanted to do a naive 'search' of $Norm_B(\mathcal{P})$ for $N\in R(\Gamma_i,[\Gamma_j])$ in which case the size of this wreath product would merit consideration.
\par
\bibliography{mp2}

\begin{thebibliography}{1}

\bibitem{Byott2004}
N.P. Byott.
\newblock Hopf-galois structures on galois field extensions of degree pq.
\newblock {\em J. Pure Appl. Algebra}, 188(1-3):45--57, 2004.

\bibitem{ChaseSweedler1969}
S.U. Chase and M.~Sweedler.
\newblock {\em Hopf Algebras and Galois Theory}.
\newblock Number~97 in Lecture Notes in Mathematics. Springer Verlag, Berlin,
  1969.

\bibitem{Childs2003}
L.N. Childs.
\newblock On hopf galois structures and complete groups.
\newblock {\em New York J. Mathematics}, 9:99--116, 2003.

\bibitem{GAP4}
The GAP~Group.
\newblock {\em {GAP -- Groups, Algorithms, and Programming, Version 4.3}},
  2002.
\newblock \verb+(http://www.gap-system.org)+.

\bibitem{GreitherPareigis1987}
C.~Greither and B.~Pareigis.
\newblock Hopf galois theory for separable field extensions.
\newblock {\em J. Algebra}, 106:239--258, 1987.

\bibitem{Kohl2013}
T.~Kohl.
\newblock Regular permutation groups of order mp and hopf-galois structures.
\newblock {\em Algebra and Number Theory}, 7(9):2203--2240, 2013.

\bibitem{Nilpotent2000}
J.~Pakianathan and K.~Shankar.
\newblock Nilpotent numbers.
\newblock {\em Amer. Math. Monthly}, 107(7):631--634, 2000.

\end{thebibliography}
\bibliographystyle{plain}
\end{document}